\documentclass{article}
\usepackage{amsfonts}
\usepackage{amsthm}
\usepackage{amsmath}
\usepackage{amssymb}
\usepackage{graphicx}
\usepackage[T1]{fontenc}
\usepackage{fullpage}

\author{
Robert Luko\v{t}ka\\Comenius University, Bratislava, Slovakia\\{\small\tt lukotka\@@dcs.fmph.uniba.sk}
}
\title{Circular flow number of Goldberg snarks}

\theoremstyle{definition}
\newtheorem*{theorem*}{Theorem}
\newtheorem{definition}{Definition}
\newtheorem{theorem}[definition]{Theorem}

\newtheorem{lemma}[definition]{Lemma}
\theoremstyle{claim}
\newtheorem{claim}{Claim}

\let\epsilon=\varepsilon

\let\phi = \varphi

\begin{document}

\maketitle


\abstract{
A \emph{circular nowhere-zero $r$-flow} on a bridgeless graph $G$ is an orientation of the edges and an assignment of real values from $[1, r-1]$ to the edges in such a way that the sum of incoming values equals the sum of outgoing values for every vertex. The \emph{circular flow number} of $G$ is the infimum over all values $r$ such that $G$ admits a nowhere-zero $r$-flow. We prove that the circular glow number of Goldberg snark $G_{2k+1}$ is $4+1/(k+1)$, proving a conjecture of Goedgebeur, Mattiolo, and Mazzuoccolo.
\smallskip

}

\bigskip


\section{Introduction}

A \emph{circular nowhere-zero $r$-flow} on a bridgeless graph $G$ is an orientation of the edges and an assignment of real values from $[1, r-1]$ to the edges in such a way that the sum of incoming values equals the sum of outgoing values for every vertex. The \emph{circular flow number}, we denote it simply by $\phi(G)$, of $G$ is the infimum over all values $r$ such that $G$ admits a nowhere-zero $r$-flow. The infimum from the definition is attained and is rational \cite{gtz}. The circular flow number is a natural refinement of the flow number of a graph, $\lceil \phi(G) \rceil$ equals the flow number of $G$ \cite{gtz}.

\emph{Snarks} are bridgeless cubic graphs that are not $3$-edge-colourable. Snarks are intensively studied mostly because many important conjectures and statements can be easily reduced to snarks, e.g. Four Colour Theorem, $5$-flow Conjecture, Cycle Double Cover Conjecture. The circular flow number of a cubic graph equals $3$ if and only if the graph is bipartite and equals $4$ if and only if the graph is non-bipartite and $3$-edge-colourable \cite{steffen, diestel}. Thus, among cubic graphs, only snarks are of interest with respect to the circular flow number and their circular flow number is more than $4$. Tutte's $5$-flow Conjecture \cite{tutte} asserts that the circular flow number of any graph is at most $5$. On the other hand any rational value between $4$ and $5$ is attainable as a circular flow number of a snark \cite{lukotka2}. 

There are three widely studied infinite classes of snarks that represent various methods of constructions: Isaac's snarks, Generalised Blanu{\v s}a snarks, and Goldberg snarks. For Isaacs snark $I_{2k+1}$ the upper bound $\phi(I_{2k+1}) \le 4+1/k$  was proved by Steffen \cite{steffen} and the corresponding lower bound was obtained by Luko{\v t}ka and {\v S}koviera \cite{lukotka}. The circular flow number of Generalized Blanu{\v s}a snarks is $4.5$ \cite{lukotka3}. The Goldberg snarks turned out to be the most challenging among these three classes. 

\emph{Goldberg snark} $G_{2k+1}$ is a graph with vertex set 
$\{a_i, b_i, c_i, d_i, e_i, f_i, g_i, \allowbreak h_i: i\in \mathbb{Z}_{2k+1}\}$ 
and edge set
$\{a_ib_i, b_ia_{i+1}, a_ie_i, b_ic_i, c_id_i, d_ie_i, c_if_i, e_ig_i, f_ig_i, \allowbreak g_if_{i+1}, d_ih_i, h_ih_{i+1}: i\in \mathbb{Z}_{2k+1}\}$ (see Figure~\ref{fig1}). 
\begin{figure}
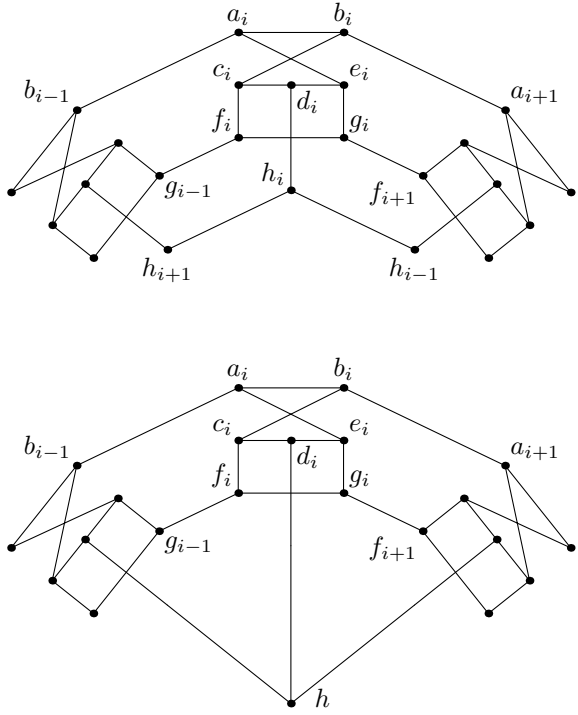
\label{fig1}
\begin{center}
\includegraphics{goldberg5.mps} \\
\includegraphics{goldberg6.mps}
\end{center}
\caption{Parts of Goldberg snark (top) and reduced Goldberg graph (bottom).}
\end{figure}
Theorem~2.4 from \cite{lukotka} implies $\phi(G_{2k+1})\ge 4+1/(2k+1)$. 
Luko{\v t}ka \cite{lukotka3} proved that $\phi(G_{2k+1})\le 4+1/k$. 
Recently, Goedgebeur, Mattiolo, and Mazzuoccolo \cite{goedgebeur} proved $\phi(G_{2k+1})\le 4+1/(k+1)$ and, based on computational evidence, conjectured that this is the correct value. In this paper we prove this conjecture.

Instead of considering Goldberg snark $G_{2k+1}$ we will consider \emph{Goldberg graph} $H_{2k+1}$ obtained by contracting all vertices in $\{h_i: i\in \mathbb{Z}_{2k+1}\}$ into a single vertex $h$ of degree $2k+1$ (see Figure~\ref{fig1}). We will show that $\phi(H_{2k+1}) \ge 4+1/(k+1)$ and thus also $\phi(G_{2k+1}) \ge 4+1/(k+1)$, completing the bound of Goedgebeur, Mattiolo, and Mazzuoccolo.

\section{Balanced valuations}

We use balanced valuations introduced by Jaeger \cite{jaegerBV}, but to avoid fractions use a definition from \cite{lukotka4}. 
Let $\partial(S)$ denote the number of edges with exactly one end in $S$.
Let $G$ be a bridgeless graph.
A \emph{balanced valuation} is a mapping $b: V(G) \to \mathbb{Z}$ such that 
\begin{itemize}
\item $b(S) \le \partial(S)$ for each $S\subseteq V(G)$ and
\item $b(v) \equiv \partial(v) (\bmod \ 2)$ for each $v\in V(G)$,
\end{itemize}
where $b(S)$ denotes the sum of $b(v)$ for all $v\in S$. To avoid zeroes in denominators we will mostly consider only subsets $S\subseteq V(G)$ with $\partial(S) \neq 0$, to write that $S$ is such subset we use $S\sqsubset G$. Note that for sets $S$ with $\partial(S) = 0$ the definition of balanced valuation implies $b(S)=0$.
A balanced valuation is \emph{orientable} if for each $S \sqsubset G$, we have $b(S) < \partial(S)$. Note that if $G$ has an orientable balanced valuation, then $G$ is bridgeless.

Given a circular nowhere-zero $r$-flow on $G$, we can create an orientable balanced valuation $b$ on $G$ by taking $b(v)$ to be equal to the number of incoming minus number of outgoing edges of $v$.
As the maximal flow value is $r-1$ one can bound the ratio of incoming and outgoing edges to $S$ and which in turn can be expressed using $b(S)$ and $\partial(S)$. The resulting condition is not only necessary, but also sufficient.
\begin{theorem}\cite{jaegerBV, lukotka4}\label{bv}
Graph $G$ has nowhere-zero circular $r$-flow if and only if there exists an orientable balanced valuation $b$ of $G$ such that for each $S \sqsubset G$,
\begin{eqnarray}
(\partial(S) + b(S))/(\partial(S) - b(S)) \le r-1. \label{eqn}
\end{eqnarray}
\end{theorem}
For $S \sqsubset G$ and an orientable balanced valuation $b$ we define $\phi(S, b) = (\partial(S) + b(S))/(\partial(S) - b(S)) +1$ and $\phi(b) = \max_{S\sqsubset G} \phi(S, b)$. Note that $\phi(G) = \min_{b \text{ is a balanced valuation of } G} \phi(b)$.
The following lemma is an easy corollary of the definitions.
\begin{lemma}\label{complement}
Let $G$ be a graph, let $b$ be an orientable balanced valuation on $G$ and let $S \sqsubset G$. Then
$\phi(S, b) = \phi(V(G)-S, -b)$ and $\phi(b) = \phi(-b)$.
\end{lemma}

The following two lemmas are implicitly present in the proof of Theorem 2.2 in \cite{lukotka}.
\begin{lemma}\label{connected1}
Let $G$ be graph, let $b$ be an orientable balanced valuation on $G$, and let $S \sqsubset G$ be such that $\phi(S, b) = \phi(b)$.
Assume that $S$ can be partitioned into two nonempty sets $A, B \sqsubset V(G)$  such that there is no edge between $A$ and $B$. Then $\phi(A, b) = \phi(B, b)= \phi(b)$.
\end{lemma}
\begin{proof}
From the definition of $\phi(b)$ we have $\phi(A, b) \le \phi(b)$ and $\phi(B, b) \le \phi(b)$. For contradiction, without loss of generality, assume $\phi(A, b) < \phi(b)$.
\begin{eqnarray*}
\phi(S, b) &=& \frac{\partial(S) + b(S)}{\partial(S) - b(S)}+1
            =  \frac{\partial(A)+\partial(B)+b(A)+b(B)}{\partial(A)+\partial(B)-b(A)-b(B)} +1\\                
           &=& \frac{\partial(A)+b(A)}{\partial(A)-b(A)+\partial(B)-b(B)} +
               \frac{\partial(B)+b(B)}{\partial(A)-b(A)+\partial(B)-b(B)} + 1\\
           &=& \phi(A, b) \frac{\partial(A)-b(A)}{\partial(A)-b(A)+\partial(B)-b(B)} \ +\\
           & &  \ \ \ \ \ \ \ \ \ \ \ \  + \ \phi(B, b) \frac{\partial(B)-b(B)}{\partial(A)-b(A)+\partial(B)-b(B)}+1 < \phi(b),
\end{eqnarray*}
a contradiction.
\end{proof}

\begin{lemma}\label{connected2}
Let $G$ be graph, $b$ be an orientable balanced valuation on $G$, and let $S \sqsubset G$ be such that $\phi(S, b) = \phi(b)$.
Assume that $V(G)-S$ can be partitioned into two nonempty sets $A, B \sqsubset V(G)$ such that there is no edge between $A$ and $B$. Then $\phi(V(G)-A, b) = \phi(V(G)-B, b)= \phi(b)$.
\end{lemma}
\begin{proof}
Follows from Lemma~\ref{connected1} by taking $-b$ as $b$, $V(G)-S$ as $S$.
\end{proof}

If a vertex $v$ has degree $3$ in $G$ (that is $v$ is not $h$), then either $b(v)=1$, we will call these vertices \emph{white}, or $b(v)=-1$, we will call these vertices \emph{black}.
This lemma is a minor modification on Remark 2.3. from \cite{goedgebeur}.
\begin{lemma}\label{lboundary}
Let $G$ be a graph, let $b$ be an orientable balanced valuation of $G$ and let $S\sqsubset G$  such that
$\phi(S, b) = \phi(b) > 3$. 
Each vertex of degree $3$ from $S$ that is neighbour with a vertex from $V(G)-S$ is white.
Each vertex of degree $3$ from $V(G)-S$ that is neighbour with a vertex from $S$ is black.
\end{lemma}
\begin{proof}
Note that it suffices to prove one of the two statements as we can obtain the other by reversing $b$ and complementing $S$. Thus, assume for contradiction, that there is a white vertex of degree $3$ from $V(G)-S$ such that $k$ of its incident edges are incident with a vertex from $S$, $k>0$. But then 
\begin{eqnarray}
\phi(S \cup \{v\})= \frac{\partial(S)+3-2k+b(S)+1}{\partial(S)+3-2k+b(S)-1} +1 > \phi(b),
\end{eqnarray}
where the last inequality is trivial for $k=1$ and $k=2$ and requires $\phi(b) > 3$ for the case $k=3$. A contradiction.
\end{proof}

Much of our technique depends on perturbing $S$. This lemma is easy corollary of the definitions.
\begin{lemma}\label{lswap}
Let $G$ be a graph, let $b$ be an orientable balanced valuation of $G$, and let $S, T \sqsubset G$.
Suppose that $b(T)=0$. Moreover, among the edges with exactly one end in $T$ exactly half of them has the other end in $S$ (and the other half in $V(G)-S$). If $T\subseteq S$ then $\phi(S-T, b) = \phi(S, b)$.
If $T \subseteq V(G)-S$, then $\phi(S \cup T, b) = \phi(S, b)$.
\end{lemma}

Finally, we generalise Theorem 2.4. from \cite{lukotka}.
\begin{lemma}\label{lflow}
Let $G$ be a graph, let $b$ be an orientable balanced valuation of $G$, and let $S \sqsubset G$.
If $4<\phi(S, b)<4+1/k$, then $\partial(S) \ge 4k+5$.
\end{lemma}
\begin{proof}
Assume $p = \partial(S) + b(S)$ and $q = \partial(S) - b(S)$. Due to the definition of balanced valuation both $p$ and $q$ are even.
We have $\phi(S, b) = p/q+1$ and thus $3q<p<3q+q/k$. Due to parity of $p$ and $q$
the first inequality yields $p\ge 3q+2$. Then $p<3q+q/k$ implies $q>2k$ and, due to parity, $q\ge 2k+2$. Therefore,  $\partial(S)=(p+q)/2 \ge 2q+1 \ge 4k+5$.
\end{proof}

\section{The proof}

For each $i\in \mathbb{Z}_{2k+1}$ the subgraph of $H_{2k+1}$ induced by $\{a_i, b_i, c_i, d_i, e_i, \allowbreak f_i, g_i\}$ is called \emph{block $i$}, all such subgraphs are called collectively \emph{blocks}. If we add $\{b_{i-1}, g_{i-1}, a_{i+1}, f_{i+1}\}$ to the vertex-set, we will call the induced subgraph \emph{extended block} $i$.
The \emph{left neighbour} of block $i$ is block $i-1$ and the \emph{right neighbour} of block $i$ is block $i+1$. A \emph{neighbour} of a block is its left or right neighbour. Note that there is an automorphism of block $i$ that takes $a_i$ to $b_i$, we will call this operation \emph{reversing} block $i$. There is an automorphism that takes $a_i$ to $g_i$, we call this operation \emph{twisting} block $i$. 

\begin{claim}\label{c4}
$4 < \phi(H_{2k+1})$
\end{claim}
\begin{proof}
Consider a $3$-edge-colouring of an extended block $i$. One can easily check the following property: the colours of $b_{i-1}a_i$ and $g_{i-1}f_i$ are equal if and only if the colours of $b_ia_{i+1}$ and $g_if_{i+1}$ are not equal. As we have odd number of blocks this implies that $H_{2k+1}-h$ has no $3$-edge-colouring. 
Assume for contradiction that $\phi(H_{2k+1}) \le 4$.
Then there is a nowhere-zero $(\mathbb{Z}_2 \times \mathbb{Z}_2)$-flow on $H_{2k+1}$ \cite{tutte}. But if we restrict the flow to $H_{2k+1}-h$ it is a $3$-edge-colouring that uses the values from $\mathbb{Z}_2 \times \mathbb{Z}_2-\{(0,0)\}$ as colours, a contradiction.
\end{proof}

For the rest of the article we consider the reduced Goldberg graph $H_{2k+1}$, where $k>0$, and denote it as $G$. 
Assume for contradiction that $\phi(G) < 4+1/(k+1)$. Let $b$ be a balanced valuation of $G$ such that $\phi(b)=\phi(G)$ (such a valuation exists because of Theorem~\ref{bv}). Let $S \sqsubset G$ such that $\phi(b,S)=\phi(G)$.
Moreover, assume that $h \in S$ (Lemma~\ref{complement}) and that $S$ induces a connected graph (Lemma~\ref{connected1}). Among all such $S$ we choose the largest one. Due to Lemma~\ref{connected2} and maximality of $S$ we have the following.
\begin{claim}\label{connectedcomplement}
$V(G)-S$ induces a connected graph.
\end{claim}
\noindent This together with the fact $h \not \in S$ greatly reduces the possibilities for the structure of the vertices in $G-S$.

\medskip

A \emph{left boundary} is a block such that its extended block satisfies
\begin{itemize}
\item $a_{i+1} \not\in S$, but $b_i, f_{i+1}\in S$, or
\item $f_{i+1} \not\in S$, but $a_{i+1}, g_i\in S$, or
\item $a_{i+1}, f_{i+1} \not\in S$, but $b_i, g_i\in S$, or
\item $b_i$ or $g_i$ is not in $S$ and there is no path containing only vertices from $V(G)-S$ between one of $b_i$ or $g_i$ to one of $a_i$, $f_i$.
\end{itemize}
\noindent A \emph{right boundary} is defined similarly using the symmetry that reverses the block. Informally speaking, a block is a left or a right boundary if the graph induced by $V(G)-S$ ends at that block. Simple case analysis based on the definition of boundary, the connectivity of graph induced by $V(G)-S$, and the fact that $h\in S$ gives the following claim.
\begin{claim}\label{boundary}
There is at most one left boundary and one right boundary in $G$.
\end{claim}

Let us start by examining $b$. Assume that there is a set $X$ containing three white (black) vertices forming a path. Then $\phi(X, b)=5$ ($\phi(V(G)-X, b)=5$), a contradiction with $\phi(b)=\phi(G)<4.5$. We can analyse other configurations $X$ that contradict $\phi(X, b) < 4.5$.
\begin{claim}\label{path}
There is no path on three vertices containing vertices of the same colour. 
\end{claim}
\begin{claim}\label{tree}
There is no tree on seven vertices containing only one vertex of some colour. 
\end{claim}
\begin{claim}\label{circuits}
There are no two $5$-circuits intersecting in three vertices containing only two vertices of some colour.
\end{claim}
\noindent An easy case analysis on the blocks of $G$ taking into account Claim~\ref{path} and Claim~\ref{circuits} yields the following. \emph{Colour switching} is an operation that replaces $b$ with $-b$.
\begin{claim}
The colours of vertices in a block match, up to colour switching and reversing, one of the configurations on Figure~\ref{confs}.
\end{claim}
\begin{figure}
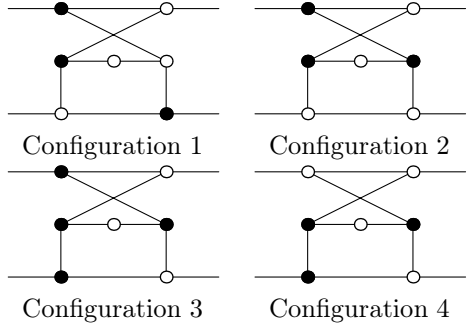

\begin{center}
\begin{tabular}{cc}
\includegraphics{goldberg1.mps} & \includegraphics{goldberg2.mps} \\
Configuration 1 & Configuration 2 \\
\includegraphics{goldberg3.mps} & \includegraphics{goldberg4.mps} \\
Configuration 3 & Configuration 4
\end{tabular}
\end{center}
\caption{Possible colouring of the blocks, up to colour switching and reverse}\label{confs}
\end{figure}
\noindent Note that Configuration 4 can be obtained from Configuration 2 by twisting and reversing the block.

Due to Claim~\ref{c4}, the assumption $\phi(G)<4+1/(k+1)$, and Lemma~\ref{lflow}.
\begin{claim}\label{largeboundary}
$\partial(S) \ge 4k+9$.
\end{claim}
\noindent We will use this claim for contradiction in the final argument, but we will also use this claim in the situations when
Claim~\ref{connectedcomplement} or the fact that $S$ induce connected graph imply that $S$ or $V(G)-S$ is small.

\bigskip

Next we analyse all possible pairs $(b, S)$ reduced to a single block $i$. 
We will call such a pair $X$ the \emph{type of the block} and we say that \emph{block $i$ is $X$}. Note that the symmetry reversing the block creates involution between block types. The type obtained by applying this involution to $X$ is the \emph{reverse} of $X$, we denote it $X^T$. We say that block $i$ is of \emph{basic type} $X$ if the type of it is $X$ or $X^T$.

We examine possible types of blocks in $G$.
In the analysis we use Lemma~\ref{lboundary}, the fact that both $S$ and $V(G)-S$ induce large (Claim~\ref{largeboundary}) connected graphs.
In addition we use Lemma~\ref{lswap} to argue maximality of $S$. We also use a combination of Lemma~\ref{lswap}, Lemma~\ref{connected1}, Lemma~\ref{connected2}, and Lemma~\ref{largeboundary}. For simple reference we capture this as this claim.
\begin{claim}\label{new}
Assume we apply Lemma~\ref{lswap} and denote $S-T$ or $S\cup T$ (according to which statement we use, $T$ is from the lemma statement) as $S'$. We may apply
Lemma~\ref{lswap} repeatedly and again name the result $S'$ (we set $S$ to be $S'$ and, again name $S-T$ or $S\cup T$ as $S'$, $T$ may be different). Then graph induced by $S$ and graph induced by $V(G-S)$ cannot have a component with boundary smaller that $4k+9$.
\end{claim}
\
Now we are ready to list all possible block types.

\begin{claim}
The possible types of blocks of $G$, up to reverse, are depicted in Figure~\ref{figc1}, Figure~\ref{figc2}, Figure~\ref{figc3}, or can be obtained by twisting the types on Figure~\ref{figc2}; vertices in $S$ are red, vertices not in $S$ are blue in the figures. Moreover, the vertices outside the block coloured red (blue) are guaranteed to belong (not to belong) to $S$.
\end{claim}
\begin{proof}
We just analyse Configuration~1, to demonstrate the techniques used. The remaining configurations are similar. Note that we do not need to consider reversed configurations as they produce reversed types. We also do not need to consider Configuration~4 as the resulting types can be obtained by twisting the types created from Configuration~2.

Consider block $i$ and assume $b$ is in Configuration 1 on it. Due to Lemma~\ref{lboundary},
$d_i, e_i\in S$. Assume, first, $c_i\in S$. Then $b_i, g_i\in S$. We get $A_1$, $B_1$, $C_1$, or $D_1$.
Next, assume $c_i\not \in S$. One of $b_i$, $f_i$ is not in $S$ due to connectivity and size of $\partial(S)$. If only one of them is not in $S$, we get a contradiction with Lemma~\ref{lswap} and maximality of $S$. Thus 
$b_i, f_i\not \in S$. But then also $a_i, g_i \not \in S$ and the type is $E_1$.

Next assume $b$ is in Configuration 1 with switched colours. If $d_i\in S$, then $a_i, c_i, e_i, g_i\in S$. We either get type $F_1$, $G_1$, $H_1$, or the case $b_i, g_i \not\in S$. But in the last case Lemma~\ref{lswap} applied to $\{c_i, d_i\}$ leads to contradiction with Lemma~\ref{lboundary}. If $d_i\not \in S$, then $e_i \not \in S$. If $c_i \in S$, then exactly one of $b_i$, $f_i$ must be in $S$. Indeed, if both are not in $S$, then $S$ induced disconnected graph contradicting the choice of $S$, and if both are in $S$, by Claim~\ref{connectedcomplement} we have $V(G)-S=\{d_i, e_i\}$ which contradicts Claim~\ref{largeboundary}. This leads to types $I_1$ and $J_1$. If $c_i \not \in S$,
then we have one of the cases $K_1$, $L_1$, $M_1$ and $N_1$.

Further analysis gives also some information about pairs $(b, S)$ on extended block $i$. We, again, focus only on types obtained from Configuration~1, and we skip types where the desired facts are implied just by Lemma~\ref{lboundary}. In pair $I_1$ we have $b_{i-1}\in S$, otherwise, applying Lemma~\ref{lswap} on $\{b_i, c_i\}$ isolates $a_i$, a contradiction with Claim~\ref{new}. Similarly $g_{i-1}\not \in S$ because otherwise Lemma~\ref{lswap} applied on $\{f_i, g_i\}$ contradicts maximality of $S$. Pair $J_1$ is just twisted $I_1$, thus the same arguments apply. For $L_1$ (and after twisting in $M_1$), we have $a_{i+1}\not \in S$ because of the maximality of $S$ (Lemma~\ref{lswap} applied on $\{b_i, c_i\}$).
\end{proof}
 
\noindent The types created from Configuration~$4$ can be obtained by twisting types  created from Configuration~2. Such a type created from $X_2$ will be denoted by $X_4$. 

\medskip

\begin{figure}
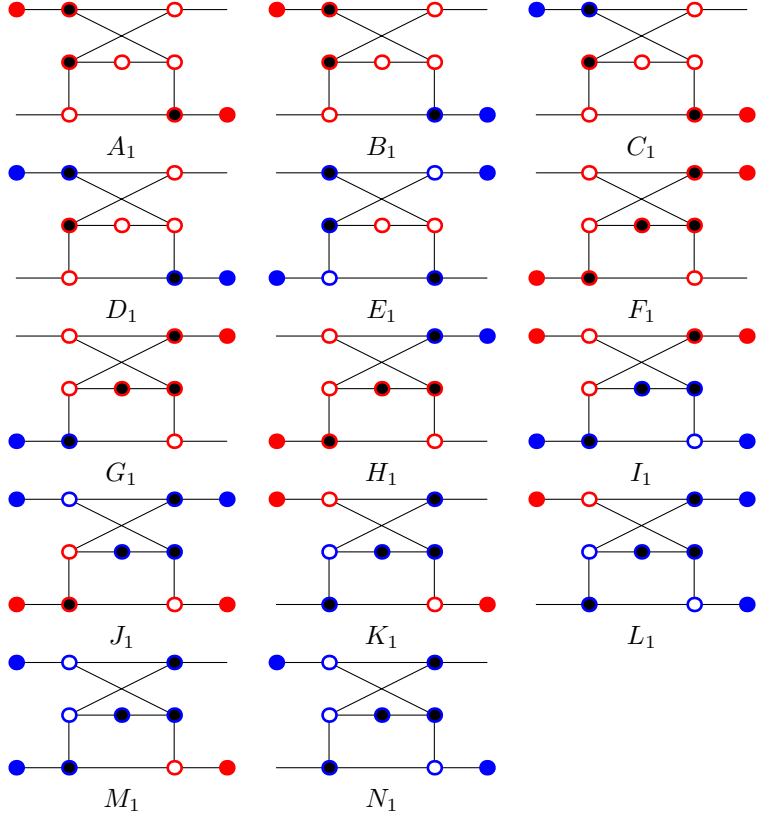

\begin{center}
\begin{tabular}{ccc}
\includegraphics{gall1.mps} &
\includegraphics{gall2.mps} &
\includegraphics{gall3.mps} \\
$A_1$ & $B_1$ & $C_1$ \\
\includegraphics{gall4.mps} &
\includegraphics{gall8.mps} &
\includegraphics{gall10.mps} \\
$D_1$ & $E_1$ & $F_1$ \\
\includegraphics{gall11.mps} &
\includegraphics{gall12.mps} &
\includegraphics{gall14.mps} \\
$G_1$ & $H_1$ & $I_1$ \\
\includegraphics{gall15.mps} &
\includegraphics{gall16.mps} &
\includegraphics{gall17.mps} \\
$J_1$ & $K_1$ & $L_1$ \\
\includegraphics{gall18.mps} &
\includegraphics{gall19.mps} & \\
$M_1$ & $N_1$ & 
\end{tabular}
\end{center}
\caption{Basic types created from Configuration~1.}\label{figc1}
\end{figure}

\begin{figure}
\begin{center}
\begin{tabular}{ccc}
\includegraphics{gall20.mps} &
\includegraphics{gall21.mps} &
\includegraphics{gall22.mps} \\
$A_2$ & $B_2$ & $C_2$ \\
\includegraphics{gall30.mps} &
\includegraphics{gall31.mps} &
\includegraphics{gall32.mps} \\
$D_2$ & $E_2$ & $F_2$ \\
\includegraphics{gall33.mps} &
\includegraphics{gall35.mps} & \\
$G_2$ & $H_2$ & 
\end{tabular}
\end{center}
\caption{Basic types created from Configuration~2.}\label{figc2}
\end{figure}

\begin{figure}
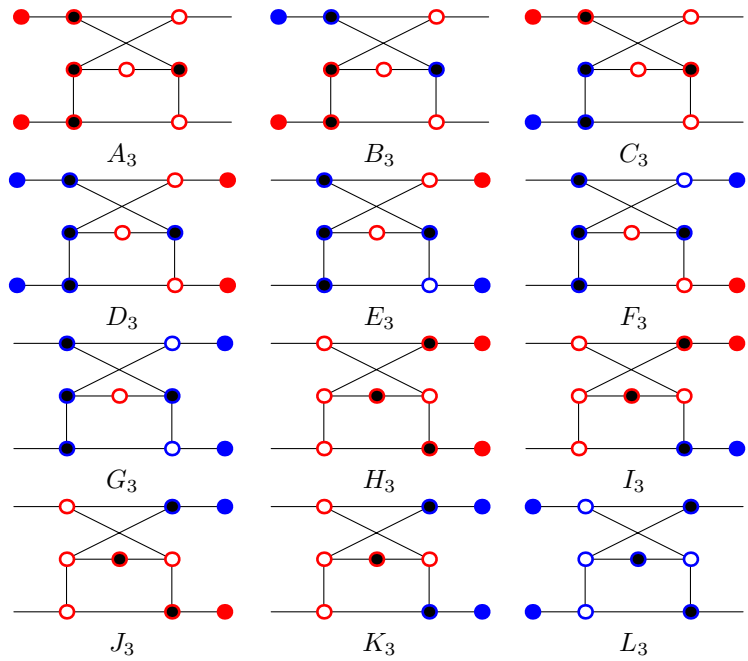

\begin{center}
\begin{tabular}{ccc}
\includegraphics{gall40.mps} &
\includegraphics{gall41.mps} &
\includegraphics{gall42.mps} \\
$A_3$ & $B_3$ & $C_3$ \\
\includegraphics{gall43.mps} &
\includegraphics{gall44.mps} &
\includegraphics{gall45.mps} \\
$D_3$ & $E_3$ & $F_3$ \\
\includegraphics{gall46.mps} &
\includegraphics{gall50.mps} &
\includegraphics{gall51.mps} \\
$G_3$ & $H_3$ & $I_3$ \\
\includegraphics{gall52.mps} &
\includegraphics{gall53.mps} &
\includegraphics{gall55.mps} \\
$J_3$ & $K_3$ & $L_3$ \\
\end{tabular}
\end{center}
\caption{Basic types created from Configuration~3.}\label{figc3}
\end{figure}

We obtain contradiction with Claim~\ref{largeboundary} using a discharging procedure. Initially, we assign charge $1$ to each edge with one endpoint in $S$ and the second in $V(G)-S$. Then move charges as follows.
\begin{enumerate}
\item Each edge sends its charge to a block containing its incident vertex that is in $S$.
\item The block whose type is one of $I_1^T$, $J_1$, $K_1$, $K_1^T$, $L_1^T$, $M_1$, $E_2$, $E_3$, $F_3$, $E_4^T$ send charge $2$ to its right neighbour,
blocks whose type is $G_2^T$, $G_4$ send $2.5$ to its right neighbour. Blocks with reverse types send the charge to its left neighbours (if both $X$ and $X^T$ is on the list, then the block sends charge to both neighbours).

\item A block with type $E_1$ or $E_1^T$ sends charge $0.5$ to both neighbours. A block with type $C_2$, $C_2^T$, $G_3$, $G_3^T$, $C_4$, $C_4^T$, receives charge from its left (right) neighbour it forwards the charge to its right (left) neighbour.
\end{enumerate}

We say that a block $i$ \emph{receives charge from left} (\emph{from right}) if it receives charge from $b_{i-1}a_i$ ($b_ia_{i+1}$), $g_{i-1}f_i$ ($g_if_{i+1}$), or its left (right) neighbour.

\begin{claim}
There is no $D_3$ or $D_3^T$ block.
\end{claim}
\begin{proof}
Suppose block $i$ is $D_3$ (the other case is obtained by reversing). As $b_{i-1}$ and $g_{i-1}$ must be white and outside $S$ block $i-1$ is one of $C_2$, $C_4^T$, $L_3^T$, or $G_3$. The first three cases contradict the maximality of $S$ with Lemma~\ref{lswap} applied on sets 
$\{c_i, f_i, f_{i-1}, g_{i-1}\}$, $\{c_i, f_i, a_{i-1}, a_{i-1}\}$, or 
$\{c_i, f_i, f_{i-1}, g_{i-1}\}$, respectively. The $G_3$ case contradicts Claim~\ref{tree} for $\{a_i,e_i, b_{i-1}, a_{i-1}, e_{i-1}, c_{i-1}, f_{i-1}\}$.
\end{proof}

\begin{claim} \label{cl01}
If a block receives charge from left (right) in Step~2, then it is the right (left) boundary. If it receives charge in Step~2 it does not receive charge in Step~3.
\end{claim}
\begin{proof}
We will consider only the case when block $i$ receives charge in Step~2 from its left neighbour and $b_{i-1} \in S$. The cases when $g_{i-1} \in S$ can be obtained by twisting and reverse; the cases when block $i$ receives charge from right neighbour can be obtained by reverse. 

Type of block $i-1$ is one of $I_1^T$, $K_1^T$, $L_1^T$, $E_2$, $G_2^T$, $E_3$. 
In all cases $g_{i-1} \not \in S$ and $b_{i-1}, a_i \in S$ and $g_{i-1}\not \in S$. 
Assume, for contradiction that block $i$ is not right boundary. By the definition of right boundary implies $f_i \not \in S$ and that there is a path of vertices from $V(G)-S$ that goes from $f_i$ to $b_i$, or $g_i$. 
The possible types for block $i$ are $I_1$, $I_1^T$, $K_1$, $L_1$
$E_2$, $E_2^T$, $G_2$, $E_3^T$. 

Assume first that the type of block $i$ is one of $E_2$, $E_2^T$, $G_2$.
But in all these cases $f_i,g_i$ are black. This contradicts Claim~\ref{path} (with $f_i, g_i, g_{i-1}$) or Claim~\ref{tree} (with $f_i, g_i, g_{i-1}$,$f_{i-1}$, $c_{i-1}$, $e_{i-1}$, $a_{i-1}$). Similar argument holds if block $i-1$ is $E_2$ or $G_2^T$, except for the case where block $i$ is $I_1^T$. To handle these cases one has to use Lemma~\ref{lswap} on $\{a_i, e_i\}$. This yields immediate contradiction with Lemma~\ref{lboundary} if block $i-1$ is $E_2$. If block $i-1$ is $G_2^T$, then we, in addition, we apply Lemma~\ref{lswap} on $\{e_{i-1}, f_{i-1}\}$. We obtain a contradiction using Claim~\ref{new}.

In the remaining cases to obtain a contradiction we use Lemma~\ref{lswap}
 on $\{b_i,c_i\}$ if the block is $I_1$, $\{a_i,e_i\}$ if the block is $I_1^T$, and do nothing otherwise. But this makes $S$ induce disconnected and  one of which is small, contradicting Lemma~\ref{connected1} and Lemma~\ref{lboundary}.
\end{proof}

Now we derive bounds on what charge can a block receive from left and right.
\begin{claim} \label{cl02}
The maximum charge a block can receive from left (right)  is $3.5$, $2.5$, or $0.5$  if $|\{a_i, f_i\} \cap S|$ ($|\{b_i, g_i\} \cap S|$) is $2$, $1$, or $0$, respectively.
\end{claim}

\begin{claim} \label{cl02b}
The maximum charge a block can receive from left (right) its left (right) neighbour is not $G_2^T$ or $G_4$ ($G_2$ or $G_4^T$) is $3$, $2$, or $0.5$  if $|\{a_i, f_i\} \cap S|$ ($|\{b_i, g_i\} \cap S|$) is $2$, $1$, or $0$, respectively.
\end{claim}

\begin{claim} \label{cl03}
The maximum charge a block can receive from left (right) if it is not the right (left) boundary is $0$, $1.5$, or $0.5$  if $|\{a_i, f_i\} \cap S|$ ($|\{b_i, g_i\} \cap S|$) is $2$, $1$, or $0$, respectively.
\end{claim}

Finally we need a claim to handle discharging in Step~3.
\begin{claim} \label{cl04}
Block with types $E_1$, $E_1^T$, $C_2$, $C_2^T$, $C_4$, $C_4^T$ cannot be neighbours.
\end{claim}
\begin{proof}
We select a set according to what the types are. For $E_i$ we pick $a_i, b_i, c_i, f_i$ if it is the right neighbour and $b_i, c_i, f_i, g_i$ if it is the left neighbour. For $E_i^T$ we pick the vertices in a reverse manner. For
$C_2$, $C_2^T$, $C_4$, or $C_4^T$ we pick vertices $a_i, b_i, c_i, e_i, f_i, g_i$. In all cases Lemma~\ref{lswap} yields contradiction with maximality of $S$.
\end{proof}

We want to prove that at the end of the procedure each block has charge at most $2$, with the exception of possible right and left boundary block that may have charge $3$ higher (thus if a block is both left and right boundary, it can have charge $8$). Considering Claim~\ref{boundary}, this shows that $\partial(S)\le 2(2k+1)+6$, which contradicts Claim~\ref{largeboundary}.

We split our investigation according type the block has.  Note that the discharging works in the same way for both pair $X$ and $X^T$, thus we do not need to consider both reverse types. Similarly, we will not consider the types created from Configuration~4 as they can be obtained by twisting. Claim~\ref{cl02} and Claim~\ref{cl03} provide sufficient bounds for most block types. We will just focus on those where these lemmas do not give sufficiently good bound on the resulting charge.

\begin{itemize}
\item $D_1, B_2, B_3, C_3$ - We can apply Claim~\ref{cl02b} instead of Claim~\ref{cl02b}.
The left neighbour cannot be $G_2^T$, or $G_4$ and the right neighbour cannot be $G_2$, or $G_4^T$ due to Claim~\ref{path}.

\item $E_1$ - 
We need to show that the block does not receive charge in Step~3 (block of type $E_1$ sends charge in Step~3, but we need to show that it cannot receive charge at the same time). Assume, for contradiction, that the block receives charge in Step~3 from right. By Claim~\ref{cl04} and the definition of discharging in Step~3 the right neighbour of block $i$ has to have basic type $G_3$.
By Claim~\ref{path} the type has to be $G_3^T$. As block $i+1$ forwards the charge to block $i$, the basic type of block $i+2$ is one of $E_1, C_2, G_3, C_4$. But Claim~\ref{path} and Claim~\ref{tree} leave only two possibilities $C_2^T$ and $C_4$. But then block $i+3$ bust again be of type $G_3^T$, and so on. But this contradicts the fact that block $i = i+2k+1$ (in $\mathbb{Z}_{2k+1}$) is of type $E_1$.

The same argument handles also the situation then $E_1$ receives charge from left (one just has to reverse the types).

\item $I_1, J_1$ - Assume the type is $I_1$ (the other case is similar). As this is neither left or right boundary, it cannot receive charge in Step~2. As $b_{i-1}, a_{i+1} \in S$ and $g_{i-1}, f_{i+1} \not \in S$ this block cannot receive charge in Step~1 or Step~3, too.

\item $G_2$ - The block cannot receive charge from left in Step~2 as it is not right boundary. As $b_{i-1} \in S$ it cannot receive charge from left altogether. By Claim~\ref{cl02} and Claim~\ref{cl03} it gets at most $0.5$ from the right. As it sends $2.5$ to its left neighbour, the resulting charge is $2$.
 \end{itemize}
 
 This shows that the final charges are as desired. This contradicts Claim~\ref{largeboundary} and finishes the proof.


\begin{thebibliography}{99}

\bibitem{diestel} R.~Diestel, Graph Theory, \emph{Electronic Edition}, Springer-Verlag New York, 2000.


\bibitem{gtz} L.~A.~Goddyn, M.~Tarsi, C.-Q.~Zhang: \emph{On $(k, d)$‐colorings and fractional nowhere‐zero flows}, Journal of Graph Theory 28 (1998), 155--161.

\bibitem{goedgebeur} J.~Goedgebeur, D.~Mattiolo, G.~Mazzuoccolo: \emph{Computational results and new bounds for the circular flow number of snarks}, Discrete Mathematics 343 (2020), 112026.

\bibitem{jaegerBV} F.~Jaeger: \emph{Balanced valuations and flows in multigraphs}, Proceedings of the American Mathematical Society 55 (1976), 237--242.

\bibitem{lukotka} R.~Luko{\v t}ka, M.~{\v S}koviera: \emph{Real flow number and the cycle rank of a graph}, Journal of Graph Theory 59 (2008), 11--16.

\bibitem{lukotka2} R.~Luko{\v t}ka, M.~{\v S}koviera: \emph{Snarks with given real flow numbers}, Journal of Graph Theory 68 (2011), 189--201.

\bibitem{lukotka3} R.~Luko{\v t}ka: \emph{Circular flow number of generalized Blanuša snarks}, Discrete Mathematics 313 (2013), 975--981.

\bibitem{lukotka4}  R.~Luko{\v t}ka: \emph{Determining the Circular Flow Number of a Cubic Graph},, Electronic Journal of Combinatorics 28 (2021), P1.49.


\bibitem{steffen} E.~Steffen: \emph{Circular flow numbers of regular multigraphs}, 
Journal of Graph Theory 36 (2001), 24--34.

\bibitem{tutte} W.~T.~Tutte: \emph{A contribution to the theory of chromatic polynomials}, J Canad Math Soc 6 (1954), 80--91.

\end{thebibliography}
\end{document}